\documentclass[12pt]{amsart}
\usepackage[a4paper,margin=2.5cm]{geometry}
\usepackage[T1]{fontenc}
\usepackage[english]{babel}
\usepackage{microtype}
\usepackage{tikz}
\usepackage{tikz-cd}
\usepackage{amsmath,amssymb,amsthm,amscd,mathtools}
\usepackage{mathrsfs}
\usepackage{enumitem}
\usepackage{xcolor}
\definecolor{citeblue}{RGB}{0,70,140}
\usepackage[
  colorlinks=true,
  linkcolor=citeblue,
  citecolor=citeblue,
  urlcolor=citeblue
]{hyperref}
\usepackage{etoolbox}

\makeatletter
\patchcmd{\@setaddresses}
  {\par\addvspace\bigskipamount\indent}
  {\par\addvspace\bigskipamount\noindent}
  {}{}
\makeatother
\makeatletter
\def\l@subsection{\@tocline{2}{0pt}{3pc}{6pc}{}}
\makeatother
\makeatletter
\renewcommand\section{\@startsection{section}{1}%
  \z@{.7\linespacing\@plus\linespacing}{.5\linespacing}%
  {\normalfont\bfseries\centering}}

\renewcommand\subsection{\@startsection{subsection}{2}%
  \z@{.5\linespacing\@plus.7\linespacing}{-.5em}%
  {\normalfont\bfseries}}
\makeatother

\numberwithin{equation}{section}

\newtheorem{theorem}{Theorem}[section]
\newtheorem{prop}[theorem]{Proposition}

\newtheorem{corollary}[theorem]{Corollary}

\newtheorem*{remark}{Remark}

\DeclareMathOperator{\Cob}{Cob}

\title[Classification of  Extended Abelian Chern--Simons Theories]{Classification of Extended\\ Abelian Chern--Simons Theories}

\author{Daniel Galviz}
\address{\noindent Yau Mathematical Sciences Center and Department of Mathematics, Tsinghua University, Beijing, China.}
\date{}

\begin{document}

\begin{abstract}
We classify extended Abelian Chern--Simons theories  with gauge group $U(1)^n$ as extended $(2+1)$-dimensional topological quantum field theories. For an even integral nondegenerate lattice $(\Lambda,K)$, let $(G_K,q_K)$ denote its discriminant quadratic module. We prove that the associated theory is determined, up to symmetric monoidal natural isomorphism, by this finite quadratic module, and that every finite quadratic module is realized as the discriminant quadratic module of an even integral nondegenerate lattice. It follows that finite quadratic modules classify extended Abelian Chern--Simons theories, pointed Abelian Reshetikhin--Turaev TQFTs, and pointed modular tensor categories.
\end{abstract}

\maketitle
\section{Introduction}
In recent work, we gave two independent constructions of Abelian Chern--Simons theories with gauge group $U(1)^n$ as extended $(2+1)$-dimensional topological quantum field theories. One is based on geometric quantization in real polarization, and the other is a rigorous functional-integral construction; the two were shown to define the same extended TQFT \cite{Galviz2,Galviz3.5}. Together with the equivalence between Abelian Chern--Simons theory and Reshetikhin--Turaev theory \cite{Galviz1,Galviz3}, this reduces the classification of Abelian Chern--Simons theories to an algebraic problem.   The aim of this paper is to solve that problem.

For an even integral nondegenerate lattice $(\Lambda,K)$ with gauge group $U(1)^n$, let
\[
(G_K,q_K), \qquad G_K=\Lambda^*/K\Lambda,
\]
denote the associated discriminant finite quadratic module. Our main result is that the corresponding Abelian Chern--Simons theory, regarded as an extended $(2+1)$-dimensional TQFT, is determined up to symmetric monoidal natural isomorphism by $(G_K,q_K)$, and that every finite quadratic module is realized in this way. Equivalently, finite quadratic modules classify  Abelian Chern--Simons theories, pointed Abelian Reshetikhin--Turaev TQFTs, and pointed modular tensor categories.

The proof has two ingredients. The first is the equivalence between Abelian Chern--Simons theory and the Reshetikhin--Turaev theory associated with the pointed modular tensor category $\mathcal{C}(G_K,q_K)$, established in \cite{Galviz1,Galviz3}. The second is a realization theorem for finite quadratic modules: Every finite quadratic module arises as the discriminant module of an even lattice; this goes back to Wall \cite[\S 6]{Wall1963}. For the explicit constructive realization used here, we follow Zhu \cite{Zhu2021}.

This result should be viewed as a classification in the toral Abelian setting; throughout this paper, Abelian Chern--Simons theory refers to the theory with gauge group a compact torus $\mathbb T\cong U(1)^n$, equivalently presented by an even integral nondegenerate lattice $(\Lambda,K)$.  Belov and Moore classified spin Abelian Chern--Simons theories with gauge group $U(1)^n$ in terms of invariants extracted from the classical lattice data, and showed that equivalence is detected by invariants of spin $3$-manifolds, or equivalently by projective modular-group representations \cite{BelovMoore}. Stirling related toral Chern--Simons theory to group categories and proved agreement of the resulting projective mapping class group representations \cite{Stirling}. On the categorical side, it is classical that pointed braided fusion categories are described by finite abelian groups equipped with quadratic forms, and that the nondegenerate case is governed by metric groups; see, for example, \cite{Joyal-Street,DGNO2010}. Wang and Wang likewise emphasize that Abelian anyon models are classified by metric groups and study explicit realization problems \cite{WangWang2020}. 

The point of the present paper is that these algebraic and representation-theoretic descriptions combine to yield a classification at the level of extended topological field theory: finite quadratic modules classify Abelian Chern–Simons theories up to symmetric monoidal natural isomorphism, rather than merely determining their closed $3$-manifold invariants, modular data, projective mapping class group representations, or underlying pointed braided categories. Thus, the present result strengthens earlier classification results by providing a complete invariant of the extended theory.

\medskip

\textbf{Acknowledgements.} I would like to thank Nicolai Reshetikhin for many helpful conversations and suggestions.

\section{Extended Theories and Equivalence Theorems}
\label{section2}

We recall the two structural results from \cite{Galviz2,Galviz3} that
enter the classification theorem.

Let $\mathbb T=\mathfrak t/\Lambda \cong U(1)^n$ be a compact torus, and let $K:\Lambda\times\Lambda \to \mathbb Z$ be an even, integral, nondegenerate symmetric bilinear form. Associated to
$(\Lambda,K)$ is the finite quadratic module
\[
G_K:=\Lambda^{*}/K\Lambda,
\qquad
q_K([x])=\exp\!\bigl(\pi i\, x^{\top}K^{-1}x\bigr).
\]
Let $\mathcal C(G_K,q_K)$ denote the corresponding pointed modular tensor
category. We write
\[
Z^{CS}_{\mathbb T,K}:\Cob^{\mathrm{ext}}_{2+1}\longrightarrow \mathrm{Vect}_{\mathbb C}
\]
for the Abelian Chern--Simons theory determined by $(\mathbb T,K)$ constructed in \cite{Galviz2}, and
\[
Z^{RT}_{\mathcal C(G_K,q_K)}:\Cob^{\mathrm{ext}}_{2+1}\longrightarrow \mathrm{Vect}_{\mathbb C}
\]
for the Walker--Maslov--corrected Reshetikhin--Turaev theory\footnote{Following \cite{Galviz2}, we use the term Reshetikhin--Turaev theory to mean the Walker--Maslov-corrected version of the theory, in which the usual Reshetikhin--Turaev construction is modified by Walker’s extension data and the Maslov correction to restore functoriality.} attached to
$\mathcal C(G_K,q_K)$ constructed in \cite{Galviz3}.

\begin{theorem}{\cite[Theorem 4.8]{Galviz2}}
Let $\mathbb T=\mathfrak t/\Lambda\cong U(1)^n$ and let
$K:\Lambda\times\Lambda\to\mathbb Z$ be even, integral, and nondegenerate.
With the $K$-twisted extended bordism convention, the assignments
\[
(\Sigma,L)\longmapsto \mathcal H_{\mathbb T,K}(\Sigma,L),
\qquad
X\longmapsto Z^{CS}_{\mathbb T,K}(X),
\]
together with the weighted extended assignment
\[
(X,L,n)\longmapsto Z^{CS}_{\mathbb T,K}(X,L,n)
:=
e^{\pi i n/4}\,
F_{L_X^{\mathbb R},L}\!\bigl(Z^{CS}_{\mathbb T,K}(X)\bigr),
\]
define a unitary extended $(2+1)$-dimensional topological quantum field theory.
\end{theorem}

In particular, for a connected closed oriented surface $\Sigma_g$ of genus $g$,
the associated state space satisfies
\[
\dim \mathcal H_{\mathbb T,K}(\Sigma_g,L)=|G_K|^g=|\det K|^g.
\]

\begin{remark}
The finite quadratic module $(G_K,q_K)$ is not merely a genus-one shadow of the
theory; it organizes the state spaces in every genus and recovers the standard
Abelian genus-one data at $\Sigma=T^2$. Thus the discriminant group controls the size of the Hilbert space in every
genus.

In genus one, let $\Sigma=T^2$, choose a symplectic basis $(a,b)$ of
$H_1(\Sigma;\mathbb Z)$, and let
\[
L_a=\mathbb R\langle a\rangle,\qquad L_b=\mathbb R\langle b\rangle.
\]
After choosing basepoints in the corresponding Bohr--Sommerfeld torsors, the
canonical bases are indexed by $G_K$, and the basic operators are given by
\[
S_{ba}(e_u):=F_{L_bL_a}(e_u)=|G_K|^{-1/2}\sum_{v\in G_K}\Omega_K(u,v)e'_v,
\qquad
T_a(e_u):=q_K(u)e_u.
\]
Hence the quantization canonically recovers the standard Abelian finite-quadratic data $(S,T,c)$ of topological order \cite{wen1990,wenzee1992,Wen2016}:
\[
\text{anyon sectors } \leftrightarrow G_K,\qquad
\text{anyon twists } \leftrightarrow q_K,\qquad
\text{mutual braiding } \leftrightarrow \Omega_K.
\]

This gives a useful interpretation of the classification theorem: the same
finite quadratic module that determines the genus-one modular data also determines the full extended theory; see \cite[Section 5]{Galviz2}.
\end{remark}

\begin{theorem}{\cite[Theorem 4.14]{Galviz3}}
\label{extended-equiv}
Let $\mathbb T=\mathfrak t/\Lambda\cong U(1)^n$ and let
$K:\Lambda\times\Lambda\to\mathbb Z$ be even, integral, and nondegenerate.
Let $(G_K,q_K)$ be the associated finite quadratic module, and let
$\mathcal C(G_K,q_K)$ be the corresponding pointed modular tensor category.
Then the Reshetikhin--Turaev theory $Z^{RT}_{\mathcal C(G_K,q_K)}$
and the Abelian Chern--Simons theory $Z^{CS}_{\mathbb T,K}$ are naturally isomorphic as symmetric monoidal extended $(2+1)$-dimensional
TQFTs. Equivalently, there exists a symmetric monoidal natural isomorphism
\[
\Phi:
Z^{RT}_{\mathcal C(G_K,q_K)}
\Longrightarrow
Z^{CS}_{\mathbb T,K}.
\]
\end{theorem}

The content of  this theorem is that the extended TQFT determined by the lattice presentation $(\Lambda,K)$ depends only on the induced finite quadratic module $(G_K,q_K)$. This is the uniqueness input needed for the classification theorem.

\begin{corollary}
Let $(\Lambda_1,K_1)$ and $(\Lambda_2,K_2)$ be even, integral, nondegenerate
lattices. If $(G_{K_1},q_{K_1})\cong (G_{K_2},q_{K_2}),$ then
\[
Z^{CS}_{\mathbb T_1,K_1}\cong Z^{CS}_{\mathbb T_2,K_2}
\]
as symmetric monoidal extended $(2+1)$-dimensional TQFTs.
\end{corollary}

\begin{proof}
If $(G_{K_1},q_{K_1})\cong (G_{K_2},q_{K_2})$, then the pointed modular tensor
categories $\mathcal C(G_{K_1},q_{K_1})$ and $
\mathcal C(G_{K_2},q_{K_2})$ are braided ribbon equivalent. Hence their  Reshetikhin--Turaev theories are naturally isomorphic as symmetric monoidal
functors. Applying the extended equivalence theorem to both lattice
presentations yields
\[
Z^{CS}_{\mathbb T_1,K_1}\cong Z^{RT}_{\mathcal C(G_{K_1},q_{K_1})}
\cong Z^{RT}_{\mathcal C(G_{K_2},q_{K_2})}
\cong Z^{CS}_{\mathbb T_2,K_2}\,\,,
\]
which proves the claim.
\end{proof}

\begin{remark}The classification theorem is genuinely a statement about extended
$(2+1)$-dimensional TQFTs, not merely about closed $3$-manifold invariants or
projective genus-one data.

At the level of closed manifolds, the comparison between Abelian Chern--Simons theory and  Reshetikhin--Turaev theory involves a residual signature phase in the raw surgery formula. This phase disappears only after passing to the Walker--Maslov--corrected extended theory. Moreover, for a bordism with boundary, the relevant closed comparison manifolds are obtained by choosing handlebody closures, and the corresponding signature defect depends on that choice. Thus equality of closed partition functions alone does not imply equality of boundary operators.

For this reason, the correct uniqueness statement for  Abelian
Chern--Simons theory is naturally formulated at the level of symmetric monoidal
extended TQFTs. It is precisely the extended equivalence theorem that shows that
the theory determined by a lattice presentation $(\Lambda,K)$ depends only on
the induced finite quadratic module $(G_K,q_K)$.
\end{remark}
\section{Classification by Finite Quadratic Modules}
%==========================================================
We now combine the equivalence theorem of Section~\ref{section2} with an explicit realization theorem for finite quadratic modules.  Recall that every finite quadratic module is isomorphic to the discriminant quadratic module of an even integral nondegenerate lattice \cite[Theorem 6]{Wall1963}. To state the realization result in a precise form,
we briefly pass to additive notation.

Let $(A,q)$ be a finite quadratic module, where $A$ is a finite abelian group and
$q \colon A \to \mathrm{U}(1)$ is a quadratic form. Write
\[
q(x)=\exp\!\bigl(2\pi i\,Q(x)\bigr),
\]
where $Q \colon A \to \mathbb{Q}/\mathbb{Z}$ is the corresponding additive quadratic form. Its associated symmetric bilinear form is
\[
b_Q(x,y):=Q(x+y)-Q(x)-Q(y)\in \mathbb{Q}/\mathbb{Z}.
\]
Orthogonal direct sums of finite quadratic modules are taken with respect to this bilinear
form. If $(\Lambda,K)$ is an even, integral, nondegenerate lattice, then $K$ induces an injective
homomorphism
\[
K \colon \Lambda \longrightarrow \Lambda^*:=\operatorname{Hom}_{\mathbb{Z}}(\Lambda,\mathbb{Z}),
\qquad
v\longmapsto K(v,-),
\]
and the associated discriminant group is $G_K:=\Lambda^*/K\Lambda.$ The induced finite quadratic form on $G_K$ is denoted by $q_K$.

\begin{prop}
\label{prop:constructive-realization}
Let $(A,q)$ be a finite quadratic module. Then there exists an even, integral, nondegenerate lattice $(\Lambda,K)$ such that
\[
(A,q)\cong (G_K,q_K),
\]
see \cite[Theorem 6]{Wall1963}. More precisely, if $q(x)=\exp(2\pi i Q(x))$, then by \cite[Definition 3.11 and Theorem 3.19]{Zhu2021},
(A,Q) admits an orthogonal decomposition into indecomposable finite quadratic modules
of the following types:
\[A^a_{p^r} =
\left(
\mathbb Z/p^r\mathbb Z,\;
x \mapsto \frac{ax^2}{p^r} + \mathbb Z
\right)
\qquad
(p \text{ an odd prime},\ r \ge 1,\ \gcd(a,p)=1),
\]
\[
A^a_{2^r} =
\left(
\mathbb Z/2^r\mathbb Z,\;
x \mapsto \frac{ax^2}{2^{r+1}} + \mathbb Z
\right)
\qquad
(r \ge 1,\ a \text{ odd}),
\]
\[
B_{2^r} =
\left(
(\mathbb Z/2^r\mathbb Z)^2,\;
(x,y)\mapsto \frac{x^2+xy+y^2}{2^r} + \mathbb Z
\right),
\]
\[
C_{2^r} =
\left(
(\mathbb Z/2^r\mathbb Z)^2,\;
(x,y)\mapsto \frac{xy}{2^r} + \mathbb Z
\right).
\]
and each such indecomposable summand is realized by an explicit even, integral,
nondegenerate lattice. Consequently, $(A,q)$ is realized by the orthogonal direct sum
of these explicit lattice blocks.
\end{prop}

\begin{proof}
Let $(A,q)$ be a finite quadratic module, and write $q(x)=\exp\!\bigl(2\pi i\,Q(x)\bigr)$ for the corresponding additive quadratic form $Q \colon A \to \mathbb{Q}/\mathbb{Z}.$
By the structure theory of finite quadratic modules, $(A,Q)$ decomposes orthogonally as a
finite direct sum of indecomposable finite quadratic modules of the four types listed above;
see \cite[Definition~3.11 and Theorem~3.19]{Zhu2021}. Thus there exists an isomorphism
of finite quadratic modules
\[
(A,Q)\cong \bigoplus_{j=1}^m M_j,
\]
where each $M_j$ is one of the following: $\quad
A_{p^r}^a,\qquad A_{2^r}^a,\qquad B_{2^r},\qquad C_{2^r}.$

For each summand $M_j$, choose an even, integral, nondegenerate lattice
$(\Lambda_j,K_j)$ whose discriminant quadratic module is isomorphic to $M_j$.
Such lattices exist explicitly by the realization theorems in \cite[Section 4]{Zhu2021}, as follows:

\begin{itemize}
\item if $M_j\cong A_{p^r}^a$ with $p$ odd, use \cite[Theorem~4.2]{Zhu2021};
\item if $M_j\cong A_{2^r}^a$, use \cite[Theorem~4.5]{Zhu2021};
\item if $M_j\cong B_{2^r}$ or $M_j\cong C_{2^r}$, use \cite[Theorem~4.7]{Zhu2021}.
\end{itemize}

Now define
\[
\Lambda:=\bigoplus_{j=1}^m \Lambda_j,
\qquad
K:=\bigoplus_{j=1}^m K_j.
\]
Since each $K_j$ is an even, integral, nondegenerate symmetric bilinear form on
$\Lambda_j$, it follows that $K$ is an even, integral, nondegenerate symmetric bilinear
form on $\Lambda$. We now compute the discriminant quadratic module of $(\Lambda,K)$.
Because duals commute with finite direct sums, there is a canonical isomorphism
\[
\Lambda^*
\cong
\bigoplus_{j=1}^m \Lambda_j^*.
\]
Under this identification, the image of $K \colon \Lambda \to \Lambda^*$
is exactly $K\Lambda
=
\bigoplus_{j=1}^m K_j\Lambda_j.$
Therefore
\[
G_K
=
\Lambda^*/K\Lambda
\cong
\bigoplus_{j=1}^m \Lambda_j^*/K_j\Lambda_j
=
\bigoplus_{j=1}^m G_{K_j}.
\]

It remains to identify the quadratic form. Let
\[
[x]=([x_1],\dots,[x_m])\in \bigoplus_{j=1}^m G_{K_j},
\]
with $[x_j]\in G_{K_j}$. Since $K$ is block diagonal, the induced discriminant quadratic
form is the orthogonal direct sum of the forms $q_{K_j}$. Concretely,
\[
q_K([x])
=
\prod_{j=1}^m q_{K_j}([x_j]).
\]

Let $Q_K$ and $Q_{K_j}$ denote the additive quadratic forms corresponding to
$q_K$ and $q_{K_j}$, respectively. Since $K$ is block diagonal, the induced
additive discriminant quadratic form is the orthogonal direct sum of the forms
$Q_{K_j}$. Therefore
\[
(G_K,Q_K)
\cong
\bigoplus_{j=1}^m (G_{K_j},Q_{K_j})
\cong
\bigoplus_{j=1}^m M_j
\cong
(A,Q).
\]
Exponentiating the additive quadratic forms then gives
\[
(G_K,q_K)\cong (A,q).
\]

Thus there exists an even, integral, nondegenerate lattice $(\Lambda,K)$ realizing the
given finite quadratic module $(A,q)$. This proves the proposition.
\end{proof}

The preceding proposition gives the existence of lattice presentations for arbitrary finite quadratic modules. The equivalence theorem of Section~\ref{section2} shows that the corresponding Abelian Chern--Simons theory depends only on the induced discriminant quadratic module. Together, these facts yield the classification theorem. The classification problem has two parts: uniqueness and realization. Uniqueness asserts that lattice presentations with isomorphic discriminant quadratic modules determine equivalent extended TQFTs, up to symmetric monoidal natural isomorphism. Realization asserts that every finite quadratic module occurs as the discriminant form of an even integral nondegenerate lattice. The classification theorem is obtained by combining these two statements.
\begin{theorem}\label{thm:classification-main}
The assignment $(\Lambda,K)\longmapsto (G_K,q_K)$ descends to a bijection from equivalence classes of presentations of  Abelian Chern--Simons theory, where
\[
(\Lambda,K)\sim (\Lambda',K')
\quad\Longleftrightarrow\quad
Z^{\mathrm{CS}}_{\mathbb T,K}\cong Z^{\mathrm{CS}}_{\mathbb T',K'}
\]
as symmetric monoidal extended $(2+1)$-dimensional TQFTs, to isomorphism classes of finite quadratic modules.

Equivalently, two presentations define equivalent  Abelian Chern--Simons theories if and only if their associated finite quadratic modules are isomorphic, and every finite quadratic module arises from some  presentation.
\end{theorem}

\begin{proof}
Let $(\Lambda,K)$ be an even, integral, nondegenerate lattice. By Theorem~\ref{extended-equiv}, the associated Abelian Chern--Simons theory $Z^{\mathrm{CS}}_{\mathbb T,K}$
is naturally isomorphic, as a symmetric monoidal functor, to the pointed Reshetikhin--Turaev TQFT associated with the finite quadratic module $(G_K,q_K)$:
\[
Z^{\mathrm{CS}}_{\mathbb T,K}
\cong
Z^{\mathrm{RT}}_{\mathcal C(G_K,q_K)}.
\]
Hence the equivalence class of the extended TQFT determined by $(\Lambda,K)$ depends only on the finite quadratic module $(G_K,q_K)$. Conversely, let $(A,q)$ be any finite quadratic module. By Proposition~\ref{prop:constructive-realization}, there exists an even, integral, nondegenerate lattice $(\Lambda,K)$ such that
\[
(A,q)\cong (G_K,q_K).
\]
Applying Theorem~\ref{extended-equiv} to $(\Lambda,K)$, we obtain an Abelian Chern--Simons theory whose equivalence class depends only on $(A,q)$. It remains to show uniqueness. Suppose $(\Lambda,K)$ and $(\Lambda',K')$ are even, integral, nondegenerate lattices with
\[
(G_K,q_K)\cong (G_{K'},q_{K'}).
\]
Then the pointed modular tensor categories $\mathcal C(G_K,q_K)$ and $\mathcal C(G_{K'},q_{K'})$ are braided ribbon equivalent. Therefore their Reshetikhin--Turaev TQFTs are naturally isomorphic as symmetric monoidal functors. Using Theorem~\ref{extended-equiv} for both lattices, it follows that
\[
Z^{\mathrm{CS}}_{\mathbb T,K}
\cong
Z^{\mathrm{CS}}_{\mathbb T',K'}
\]
as symmetric monoidal extended $(2+1)$-dimensional TQFTs. Thus finite quadratic modules classify Abelian Chern--Simons theories up to symmetric monoidal natural isomorphism.
\end{proof}

\begin{corollary}
\label{cor:equivalent-classifications}
Isomorphism classes of finite quadratic modules are in bijection with equivalence classes of:
\begin{enumerate}
\item[\textup{(i)}] pointed modular tensor categories in the  Abelian setting;
\item[\textup{(ii)}] pointed Abelian extended  Reshetikhin--Turaev TQFTs;
\item[\textup{(iii)}] Abelian extended Chern--Simons TQFTs.
\end{enumerate}
\end{corollary}

\begin{proof}
The classification of pointed modular tensor categories by finite quadratic modules is standard \cite{Joyal-Street, DGNO2010}. Braided ribbon equivalent modular categories determine naturally isomorphic symmetric monoidal Reshetikhin--Turaev TQFTs \cite{Reshetikhin:1991,Turaev1994}, and by Theorem~\ref{extended-equiv} these are naturally isomorphic to the corresponding Abelian Chern--Simons theories. The result now follows from Theorem~\ref{thm:classification-main}.
\end{proof}

\begin{theorem}
Let $Z:\Cob^{\mathrm{ext}}_{2+1}\longrightarrow \mathrm{Vect}_{\mathbb C}$ be an extended Abelian TQFT in the sense that $Z$ is naturally isomorphic to
the Reshetikhin--Turaev TQFT of a pointed modular
tensor category. Then there exists a finite quadratic module $(G,q)$ and an
even, integral, nondegenerate lattice $(\Lambda,K)$ such that
\[
Z \cong Z^{RT}_{\mathcal C(G,q)}
\cong Z^{CS}_{\mathbb T,K},
\qquad
\mathbb T=\mathfrak t/\Lambda.
\]
In particular, equivalence classes of such extended Abelian TQFTs are classified
by isomorphism classes of finite quadratic modules.
\end{theorem}

\begin{proof}
Let $\mathcal C$ be a pointed modular tensor category such that $Z\cong Z^{RT}_{\mathcal C}.$
By the standard classification of pointed modular tensor categories, there
exists a finite quadratic module $(G,q)$ such that
$\mathcal C \simeq \mathcal C(G,q)$ as braided ribbon categories. Hence
\[
Z\cong Z^{RT}_{\mathcal C(G,q)}.
\]
By Proposition~\ref{prop:constructive-realization}, choose an even, integral, nondegenerate lattice
$(\Lambda,K)$ with $(G,q)\cong (G_K,q_K)$. Then Theorem~\ref{extended-equiv} gives
\[
Z^{RT}_{\mathcal C(G_K,q_K)}\cong Z^{CS}_{\mathbb T,K}.
\]
Since $(G,q)\cong (G_K,q_K)$, the corresponding pointed modular tensor
categories are braided ribbon equivalent, so
\[
Z^{RT}_{\mathcal C(G,q)}\cong Z^{RT}_{\mathcal C(G_K,q_K)}.
\]
Combining these isomorphisms yields
\[
Z\cong Z^{RT}_{\mathcal C(G,q)}\cong Z^{CS}_{\mathbb T,K},
\]
as claimed.
\end{proof}

\begin{remark}
The classification theorem may be summarized by the following diagram:
\[
\begin{tikzpicture}[baseline=(current bounding box.center), every node/.style={align=center}]
  \node (A) at (0,2) {\text{Finite quadratic}\\\text{modules}};
  \node (B) at (6,2) {\text{Pointed modular}\\\text{tensor categories}};
  \node (C) at (6,0) {\text{Pointed Abelian}\\\text{Reshetikhin--Turaev TQFTs}};
  \node (D) at (0,0) {\text{ Abelian}\\\text{Chern--Simons TQFTs}};

  \draw[<->, thick] (A) -- (B);
  \draw[<->, thick] (B) -- (C);
  \draw[<->, thick] (C) -- (D);
  \draw[<->, thick] (D) -- (A);
\end{tikzpicture}
\]
The point is that the same Abelian theory admits four equivalent descriptions. A lattice presentation $(\Lambda,K)$ determines a finite quadratic module $(G_K,q_K)$; this quadratic data determines a pointed modular tensor category and hence a pointed Abelian Reshetikhin--Turaev TQFT; and by the extended equivalence theorem this is precisely the corresponding  Abelian Chern--Simons theory. Conversely, every finite quadratic module is realized by an even integral nondegenerate lattice. Thus finite quadratic modules are the intrinsic invariants classifying  Abelian Chern--Simons theories up to symmetric monoidal natural isomorphism.
\end{remark}

\bibliographystyle{alpha}
\renewcommand{\refname}{References}
\bibliography{refs}
\end{document}